\documentclass{amsart}
\usepackage{graphicx} 
\usepackage{times,colordvi,amsmath,epsfig,float,multicol,subfigure} 
\usepackage{color}


\newcommand{\R}{\mathbb{R}}

\DeclareMathAlphabet{\mathpzc}{OT1}{pzc}{m}{it}

\newcommand{\Mundo}{\mathfrak{X}^{1}(M)}

\newcommand{\Cmundo}{\mathfrak{X}_\mu^1(M)}

\newcounter{main}

\newtheorem{theorem}{Theorem}
\newtheorem*{theorema}{Theorem A}
\newtheorem*{theoremaa}{Theorem A$^\prime$}
\newtheorem*{theoremb}{Theorem B}
\newtheorem*{theorembb}{Theorem B$^\prime$}
\newtheorem*{theoremc}{Theorem C}

\newtheorem{proposition}[theorem]{Proposition}
\newtheorem{lemma}[theorem]{Lemma}
\newtheorem{corollary}[theorem]{Corollary}

\newcommand{\blanksquare}{\,\,\,$\sqcup\!\!\!\!\sqcap$}

\newcounter{example}
{{\stepcounter{example}}{\flushleft {\bf Example \arabic{example}:}}}%
{\par}
\DeclareMathOperator{\Crit}{Crit}

\title[Conservative flows with various types of  shadowing]
{Conservative flows with  various types of  shadowing}

\author[M. Bessa]{M\'{a}rio Bessa}
\address{Universidade da Beira Interior, Rua Marqu\^es d'\'Avila e Bolama,
  6201-001 Covilh\~a,
Portugal}
\email{bessa@ubi.pt}

\author[R. Ribeiro]{Raquel Ribeiro}
\address{Instituto de Matem\'atica, Universidade Federal do Rio de
Janeiro, P.O. Box 68530, Rio de Janeiro 21945-970, Brazil}
\email{ribeiro@impa.br}

\begin{document}

\begin{abstract}
In the present paper we study  the $C^1$-robustness of the three properties: average shadowing, asymptotic average shadowing and
limit shadowing within two classes of conservative flows: the
incompressible and the Hamiltonian ones.  We obtain that the first two properties guarantee
dominated splitting (or partial hyperbolicity) on the whole manifold, and the third
one implies that the flow is Anosov.
\end{abstract}

\maketitle



\section{Introduction: basic definitions and statement of the results}

It is known since long time ago that nonlinear systems behave, in general, in a quite complicated fashion. One of the most fundamental example of that was given by Anosov when studying the geodesic flow associated to metrics on manifolds of negative curvature (\cite{A}). Anosov obtained a striking geometric-dynamical behavior, now called \emph{uniform hyperbolicity}, of those systems in particular a global form of uniform hyperbolicity (a.k.a. Anosov flows). The core characteristic displayed by uniform hyperbolicity is, in brief terms, that on some invariant directions by the tangent flow we observe uniform contraction or expansion along orbits, and these rates are uniform. In the 1960's the hyperbolicity turned out to be the main ingredient which trigger the construction of a very rich theory of a wide class of dynamical systems (see e.g. \cite{S1,KH}). It allows us to obtain a fruitful ge\-o\-me\-tric theory (stable/unstable manifolds), a stability theory (in rough terms that hyperbolicity is tantamount to structural stability), a statistical theory (smooth ergodic theory) and a numerical theory (shadowing and expansiveness) are some examples of powerful applications of the uniform hyperbolicity concept. However, from an early age one began to understand that the uniform hyperbolicity was far from covering all types of dynamical systems and naturally other more relaxed definitions began to emerge (nonuniform hyperbolicity, partial hyperbolicity and dominated splitting, see e.g. ~\cite{BDV}). As mentioned above, the hyperbolicity was found to contain very interesting numerical properties. Actually, the hyperbolic systems display the \emph{shadowing property}: meaning that quasi-orbits, that is, almost orbits affected with a certain error, were shaded by true orbits of the original system. This amazing property, which is not present in partial hyperbolicity (see~\cite{BDT}), contained itself much of the typical rigidity of the hyperbolicity and its strong assumptions. Nonetheless, a much more surprising fact is that, under a certain stability hypothesis, the other way around turns out to be also true. To be more precise, if we assume that we have the robustness of the shadowing property, then the dynamical system is uniformly hyperbolic. In overall, some stability of a pointedly numerical property, allow us to obtain a geometric, dynamic and also topological feature. The next step then was to address the following question: \emph{is it possible to weaken the shadowing property and obtain the same conclusions?} If not, \emph{how far we can get in our findings?}
In the present paper we deal with three enfeebled branches of
shadowing: the average shadowing, the asymptotic average shadowing
and the limit shadowing (see \S\ref{SP} for full details). In
conclusion, we prove that the stability of these types of
shadowing for conservative flows imply (some) hyperbolicity. More specifically, the
stability of the first two  types of shadowing mentioned above
imply that the flow admits a dominated splitting in the whole
manifold, and the one of the third  shadowing guarantees that the
flow is of Anosov type. Theses results hold for incompressible flows and Hamiltonian
ones and in arbitrarily high dimension. See \S\ref{theend} for the  statements
of the main results of this work.

\begin{subsection}{Dissipative and incompressible flows setting}\label{CI}
Along this paper we consider vector fields $X:M\rightarrow{TM}$,
where $M$ is a $d$-dimensional ($d\geq 3$) connected and closed
$C^{\infty}$ Riemannian manifold $M$ and $TM$ its tangent bundle.
Given a vector field $X$ we have an associated flow $X^{t}$ which
is the infinitesimal generator of $X$ in a sense that
$\partial_t X^{t}|_{t=s}(p)=X(X^{s}(p))$. If the divergence of
$X$, defined by $\nabla \cdot X=\sum_{i=1}^d \frac{\partial
X_i}{\partial x_i}$, is zero we say that $X$ is divergence-free.
The flow $X^{t}$ has a tangent flow $DX^{t}_{p}$ which is the
solution of the non-autonomous linear variational equation
$\partial_t u(t)=DX_{X^{t}(p)}\cdot u(t)$. Moreover, due to Liouville's
formula, if $X$ is divergence-free, the associated flow $X^t$
preserves the volume-measure and for this reason we call it
\emph{incompressible}. If the vector field is not divergence-free
its flow is \emph{dissipative}. We denote by $\mathfrak{X}^{1}(M)$
the set of all dissipative
 $C^1$ vector fields  and by  $\mathfrak{X}^{1}_{\mu}(M)\subset
\mathfrak{X}^{1}(M)$  the set of all $C^{1}$ vector fields that
preserve the volume, or equivalently the set of all incompressible
flows. We assume that both $\mathfrak{X}^{1}(M)$ and
$\mathfrak{X}^{1}_{\mu}(M)$ are endowed with the $C^{1}$ Whitney
(or strong) vector field topology which turn these two vector
spaces completed, thus a Baire space.  We denote by $\mathcal{R}$
the set of regular points of $X$, that is, those points $x$ such
that $X(x)\not=\vec0$ and by Sing ($X$)$=M\setminus\mathcal{R}$
the set of singularities of $X$. Let us  denote by $\Crit(X)$ the
set of critical orbits of $X$, that is, the set formed by all
periodic orbits and all singularities of $X$.

The Riemannian structure on $M$ induces a norm $\left\|\cdot\right\|$ on the fibers $T_p M$, $\forall \:p\in M$. We will use the standard norm of a bounded linear map $L$ given by
\begin{equation}
\displaystyle\left\|L\right\|=\sup_{\left\|u\right\|=1}\left\|L(u)\right\|.\nonumber
\end{equation}

A metric on $M$ can be derived in the usual way by using the
exponential map or through the Moser volume-charts (cf. \cite{Mo})
in the case of volume manifolds, and it will be denoted by
$d(\cdot,\cdot)$. Hence, we define the open balls $B(x,r)$ of the
points $y\in M$ satisfying $d(x,y)<r$ by using those charts.

Dissipative flows appear often in models given by differential
equations in mathema\-tical physics, economics, biology, engineering
and many diverse areas. Incompressible flows arise naturally in
the fluid mechanics formalism and has long been one of the most
challenging research fields in ma\-the\-ma\-ti\-cal physics.
\end{subsection}

\begin{subsection}{The Hamiltonian flow formalism}
Let $(\emph{M},\omega)$ be a compact symplectic manifold, where
$\emph{M}$ is a $2d$-dimensional ($d\geq 2$), smooth and compact
Riemannian manifold endowed with a symplectic structure $\omega$,
that is, a skew-symmetric and nondegenerate 2-form on the tangent
bundle $T\emph{M}$. We notice that we use the same notation for
manifolds supporting Hamiltonian flows and also flows as in
 Subsection \ref{CI}, which we  hope will not be ambiguous.

We will be interested in the Hamiltonian dynamics of real-valued
$C^2$ functions on $\emph{M}$, constant on each connected
component of the boundary of $\emph{M}$, called
\emph{Hamiltonians}, whose set we denote by
$C^2(\emph{M},\mathbb{R})$.  For any Hamiltonian function
$H\colon\emph{M}\longrightarrow \mathbb{R}$ there is a
corres\-pon\-ding \textit{Hamiltonian vector field}
$X_H\colon\emph{M}\longrightarrow T\emph{M}$,  tangent to the
boundary of $\emph{M}$, and determined by the equality
\begin{displaymath}
\nabla_pH(u)=\omega(X_H(p),u), \;\forall u\in T_p\emph{M},
\end{displaymath}
where $p\in \emph{M}$.

Observe that $H$ is $C^2$ if and only if $X_H$ is $C^{1}$. Here we
consider the space of the Hamiltonian vector fields endowed with
the $C^1$ topology, and for that we consider
$C^2(\emph{M},\mathbb{R})$ equipped with the $C^2$ topology.

The Hamiltonian vector field $X_H$ generates the
\textit{Hamiltonian flow} $X_H^t$, a smooth 1-parameter group of
symplectomorphisms on $\emph{M}$ sa\-tis\-fying
$\partial_t X_H^t=X_H (X_H^t)$ and $X_H^0=id$. We also consider
the \emph{tangent flow} $D_pX_H^t\colon T_p\emph{M}\longrightarrow
T_{X_H^t(p)}\emph{M}$, for $p\in \emph{M}$, that satisfies the
linearized differential equality
$\partial_t D_pX_H^t=(D_{X_H^t(p)}X_H)\cdot D_pX_H^t$, where
$D_pX_H\colon T_p\emph{M}\longrightarrow T_p\emph{M}$.

Since $\omega$ is non-degenerate, given $p \in \emph{M}$,
$\nabla_pH=0$ is equivalent to $X_H(p)=0$, and we say that $p$ is
 a \textit{singularity} of $X_H$. A point is said to be
\textit{regular} if it is not a singularity. We denote by
$\mathcal{R}$ the set of regular points of $H$, by Sing ($X_H$)
the set of singularities of $X_H$ and by Crit ($H$) the set of
critical orbits of $H$.

By the theorem of Liouville (\cite[Proposition 3.3.4]{AbM}), the
symplectic manifold $(M,\omega)$ is also a volume manifold, that
is, the $2d$-form
$\omega^d=\omega\wedge\overset{d}{...}\wedge\omega$ is a volume
form and induces a measure $\mu$ on $M$, which is called the
Lebesgue measure associated to $\omega^d$. Notice that the measure
$\mu$ on $M$ is invariant by the Hamiltonian flow.

Fixed a Hamiltonian $H\in C^2(\emph{M},\mathbb{R})$ any scalar
$e\in H(\emph{M})\subset \mathbb{R}$ is called an \textit{energy}
of $H$ and $H^{-1}(\left\{e\right\})=\left\{p\in \emph{M}:
H(p)=e\right\}$ is the corresponding \textit{energy level} set
which is $X_H^t$-invariant. An \textit{energy surface}
$\mathcal{E}_{H,e}$ is a connected component of
$H^{-1}(\left\{e\right\})$; we say that it is \emph{regular} if it
does not contain singularity points and in this case
$\mathcal{E}_{H,e}$ is a regular compact $(2d-1)$-manifold.
Moreover, $H$ is constant on each connected component
$\mathcal{E}_{H,e}$ of the boundary $\partial \emph{M}$.

A {\em Hamiltonian system} is a triple $(H,e, \mathcal{E}_{H,e})$, where $H$ is a Ha\-mil\-ton\-ian, $e$ is an energy and $\mathcal{E}_{H,e}$ is a regular connected component of $H^{-1}(\{e\})$.

Due to the compactness of $\emph{M}$, given a Hamiltonian function
$H$ and $e\in H(\emph{M})$ the e\-nergy level
$H^{-1}(\left\{e\right\})$ is the union of a finite number of
disjoint compact connected components, separated by a positive
distance.  Given $e \in H(\emph{M})$,  the pair $(H,e) \subset
C^2(\emph{M},\mathbb{R}) \times \mathbb{R}$ is called a
\textit{Hamiltonian level}; if we fix $\mathcal{E}_{H,e}$ and a
small neighbourhood $\mathcal{W}$ of $\mathcal{E}_{H,e}$ there
exist a small neighbourhood $\mathcal{U}$ of $H$ and $\delta>0$
such that for all $\tilde{H} \in \mathcal{U}$ and $\tilde{e} \in
]e-\delta,e+\delta[$ one has that
$\tilde{H}^{-1}(\{\tilde{e}\})\cap
\mathcal{W}=\mathcal{E}_{\tilde{H},\tilde{e}}$. We call
$\mathcal{E}_{\tilde{H},\tilde{e}}$ the \emph{analytic
continuation} of $\mathcal{E}_{H,e}$.

Using the Darboux charts (cf. \cite{MZ}) we define a metric on
$\emph{M}$ which we also denote by $d(\cdot,\cdot)$. Let $B(x,r)$
stand for the open balls centered in $x$ and with radius $r$ by
using Darboux's charts.

The Hamiltonian formalism appears in various branches of pure and applied mathe\-matics and, due to its ubiquity, it is completely undeniable the importance and impact of this fundamental concept in science today. We refer the book \cite{MZ} for a full detailed exposition about Hamiltonian formalism.

\end{subsection}


\begin{subsection}{Properties of the Shadowing}\label{SP}

The concept of shadowing in dynamical systems has both
applications in numerical theoretical analysis and also to
structural stability and hyperbolicity. In rough terms shadowing
is supported in the idea of estimating differences between exact
and approximate solutions along orbits and to understand the
influence of the errors that we commit and allow on each iterate.
We usually ask if it is possible to obtain shadowing of ``almost"
trajectories in a given dynamical system by exact ones.

It is interesting to take a more general context where the errors of the
``almost" trajectories can be large, however, on \emph{average}
they remain small. This concept, much more relaxed than shadowing,
and called average shadowing was introduced by Blank ~\cite{BA}
 and is one of the main subjects of this work.

A sequence $(x_i,t_i)_{i=0}^m$, with $m\in\mathbb{Z}$, is called a
$\delta$-\emph{pseudo-orbit} of a given vector field $X$
(dissipative, incompressible or Hamiltonian) if for every $0\leq
i \leq m-1$ we have $t_i\geq 1$ and
$$d(X^{t_i}(x_i),x_{i+1})<\delta.$$

A sequence $(x_i,t_i)_{i\in\mathbb{Z}}$ is a
$\delta$-\emph{average-pseudo-orbit of $X$},  if $t_i\geq 1$ for
every $i\in\mathbb{Z}$ and there is a number $N$ such that for any $n\geq
N$ and $k\in\mathbb{Z}$ we have
$$\frac{1}{n}\sum_{k=1}^{n} d(X^{t_{i+k}}(x_{i+k}),x_{i+k+1})<\delta.$$

A sequence $(x_i,t_i)_{i\in\mathbb{Z}}$ is \emph{positively
$\epsilon$-shadowed in average} by the orbit of $X$ through a
point $x$, if there exists an orientation preserving homeomorphism
$h \colon \mathbb{R}\rightarrow \mathbb{R}$ with $h(0)=0$ such
that
\begin{equation}\label{br02}\underset{n\rightarrow
\infty}{\limsup}\frac{1}{n}\sum_{i=1}^n \int_{s_i}^{s_{i+1}}
d(X^{h(t)}(x), X^{t-s_i}(x_i))dt<\epsilon,
\end{equation}

\noindent where $s_0=0$, $s_n=\sum_{i=0}^{n-1}t_i$ and
$s_{-n}=\sum_{i=-n}^{-1}t_i$ for $n\in \mathbb{N}$. Analogously,
we define \emph{negatively $\epsilon$-shadowed in average}.

We say that $X$ has the \emph{average shadowing property} if for
any $\epsilon>0$, there is a $\delta>0$ such that any
$\delta$-average-pseudo-orbit,  $(x_i,t_i)_{i\in\mathbb{Z}}$,  of
$X$ can be positively and negatively $\epsilon$-shadowed in
average by some orbit of $X$.


In the other hand, the authors in  \cite{ENS} posed the notion of
the limit-shadowing property. From the numerical point of view
this property on a dynamical system $X$ means that if we apply a
numerical method of approximation to $X$ with ``improving
accuracy" so that one step errors tend to zero as time goes to
infinity then the numerically obtained trajectories tend to real
ones. Such situations arise, for example, when one is not so
interested on the initial (transient) behavior of trajectories but
wants to reach areas where ``interesting things" happen (e.g.
attractors) and then improve accuracy. To be more precise, we say
that a sequence $(x_i,t_i)_{i\in \mathbb{Z}}$ is a
\textit{limit-pseudo orbit} of $X$ (dissipative, incompressible or
Hamiltonian) if $t_i\geq1$ for every $i\in \mathbb{Z}$ and
\begin{equation*}\label{quesaco10}
    \lim_{|i|\rightarrow\infty}d(X^{t_i}(x_i),x_{i+1})=0.
\end{equation*}
A limit-pseudo orbit $(x_i,t_i)_{i\in \mathbb{Z}}$ of $X$ is
\textit{positively shadowed in limit} by an orbit of $X$ through a
point $x$, if there is an orientation preserving homeomorphism
${h}: \R \rightarrow \R$ with $h(0)=0$  such
that
\begin{equation}\label{quesaco11}
    \lim_{i\to\infty}\int_{s_i}^{s_{i+1}}d(X^{h(t)}(x),
    X^{t-{s_i}}(x_i))\,dt=0.
\end{equation}
Analogously, as we did before, we define when a limit-pseudo orbit
is said to be \textit{negatively shadowed in limit} by an orbit.

We say that $X$ has the \emph{limit shadowing property} if any
limit pseudo-orbit, $(x_i,t_i)_{i\in\mathbb{Z}}$, of $X$ can be
positively and negatively shadowed in limit by some orbit of $X$.


Finally, Gu~\cite{GSX2} introduced the notion of the asymptotic
average shadowing property for flows which is particularly well
adapted to random dynamical systems. A sequence $(x_i,t_i)_{i\in
\mathbb{Z}}$ is an \textit{asymptotic average-pseudo orbit} of $X$
(dissipative, incompressible or Hamiltonian) if $t_i\geq 1$ for
every $i\in \mathbb{Z}$ and
\begin{equation*}
    \lim_{n\to \infty}\frac{1}{n}\sum_{i=-n}^nd(X^{t_i}(x_i),x_{i+1})=0.
\end{equation*}
A sequence  $(x_i,t_i)_{i\in \mathbb{Z}}$ is \textit{positively
asymptotically shadowed in average} by an orbit of $X$ through
$x$, if there exists an orientation preserving homeomorphism ${h}:
\R \rightarrow \R$ such that
\begin{equation} \label{saco1000}
    \lim_{n\to\infty}\frac{1}{n}\sum_{i=0}^n\int_{s_i}^{s_{i+1}}d(X^{h(t)}(x), X_{t-s_i}(x_i))dt=0.
\end{equation}

Again, where $s_0=0$ and $s_n=\sum_{i=0}^{n-1}t_i$, $n\in
\mathbb{N}$. Similarly  an asymptotic average-pseudo orbit is
\textit{negatively asymptotically shadowed in average}.

We say that $X$ has the \emph{asymptotic average shadowing
property} if for any $\epsilon>0$, any asymptotic
average-pseudo-orbit,  $(x_i,t_i)_{i\in\mathbb{Z}}$, of $X$ can be
positively and negatively asymptotically shadowed  in average by
some orbit of $X$.


Note that the definitions of shadowing  above allows the presence
of a re\-pa\-ra\-metrization of the trajectory. In this case we
have the difficulty of analyzing the existence (or not) of an
orbit that shadows the pseudo orbit, because we can not control
how fast the orbit moves. In fact, the reparametrization allows,
in a short time, an orbit possibly distant pseudo orbit, approach
of it so that the limits (\ref{br02}), (\ref{quesaco11}) and
(\ref{saco1000}) are satisfied. This does not occur when the
function $h$ is equal to the identity, for example, as in
\cite{Ba}.

We observe also that the above shadowing concepts  are not
equivalent. Recall that Morse-Smale vector fields,  are a class
within the dissipative flows, admiting sinks and sources. It is known that the average
shadowing, the asymptotic average shadowing, and the limit
shadowing properties each imply that there are neither sinks nor
sources in the system. Thus, a Morse-Smale vector field is an
example of a vector field which has the shadowing property but do
not have average shadowing property, or asymptotic average
shadowing property or limit shadowing property. See \cite{Niu} for
more examples. Examples of systems which have the asymptotic
average shadowing property or the limit shadowing property, but do
not have the shadowing property can found in \cite{GSX3} and
\cite{P}, respectively.

Finally, we define the set of \emph{$C^1$-stably average shadowing
flows}:

\begin{itemize}
\item If  $X\in \mathfrak{X}^1(M)$ we say that $X^t$ is a
\emph{$C^1$-stably average shadowable flow} if any $Y\in
\mathfrak{X}^1(M)$ sufficiently $C^1$-close to $X$  has the
average shadowing property;

\item If  $X\in \mathfrak{X}^1_\mu(M)$ we say that $X^t$ is a
\emph{$C^1$-stably average shadowable incompressible flow }if any
$Y\in \mathfrak{X}^1_\mu(M)$ sufficiently $C^1$-close to $X$ has
the average shadowing property and

\item The Hamiltonian system $(H,e, \mathcal{E}_{H,e})$ is {\em
stably average shadowable} if there exists a neighbourhood
$\mathcal{V}$ of $(H,e, \mathcal{E}_{H,e})$ such that any
$(\tilde{H},\tilde{e}, \mathcal{E}_{\tilde{H},\tilde{e}}) \in
\mathcal{V}$ has the average shadowing property.
\end{itemize}
Analogously we define the sets of \emph{$C^1$-stably asymptotic
average shadowable flows} and \emph{$C^1$-stably limit shadowable
flows}.

Here we will include in our studies flows that in general can have
singularities. \emph{Posteriori} we will derive that the three
properties of shadowing in fact imply the absence of
singularities. We refere  \cite{ARR}  and  \cite{K}
 for
a discussion of types of shadowing in the Lorenz flow containing
one singularity.
\end{subsection}


\begin{subsection}{Hyperbolicity and statement of the results}\label{theend}
Let $\Lambda \subseteq \mathcal{R}$ be an $X^t$-invariant set. We
say that $\Lambda$ is \emph{hyperbolic} with respect to the vector
field $X\in\mathfrak{X}^1(M)$ if, there exists $\lambda\in(0,1)$
such that, for all $x\in \Lambda$, the tangent vector bundle over
$x$ splits into three $DX^t(x)$-invariant subbundles $T_x
\Lambda=E^u_x\oplus E_x^0\oplus E_x^s$, with
$\|DX^1(x)|_{E^s_x}\|\leq \lambda$ and
$\|DX^{-1}(x)|_{E_x^u}\|\leq \lambda$ and $E_x^0$ stands for the
one-dimensional flow direction. If $\Lambda=M$ the vector field
$X$ is called \emph{Anosov}. We observe that there are plenty
Anosov flows which are not incompressible. Despite the fact that
all incompressible Anosov flows are transitive, there exists
dissipative ones which are not (see ~\cite{FW}).

Given $x\in \mathcal{R}$ we consider its normal bundle $N_{x}=X(x)^{\perp}\subset T_{x}M$ and define the \emph{linear Poincar\'{e} flow} by $P_{X}^{t}(x):=\Pi_{X^{t}(x)}\circ DX^{t}_{x}$ where $\Pi_{X^{t}(x)}:T_{X^{t}(x)}M\rightarrow N_{X^{t}(x)}$ is the projection along the flow direction $E^0_x$. Let $\Lambda \subset \mathcal{R}$ be an $X^t$-invariant set and $N=N^{1}\oplus  N^{2}$ be a $P_{X}^{t}$-invariant splitting over $\Lambda$ with $N^1$ and $N^2$ one-dimensional. Fixed $\lambda\in(0,1)$ we say that this splitting is an $\lambda$-\emph{dominated splitting} for the linear Poincar\'{e} flow if for all $x\in \Lambda$  we have:
$$\|P_{X}^{1}(x)|_{N^{2}_{x}}\|  . \|P_{X}^{-1}(X^1(x))|_{N^{1}_{X^1(x)}}\|   \leq \lambda.$$

This definition is weaker than hyperbolicity where it is required that
\begin{equation}\label{lpf}
\|P_{X}^{1}(x)|_{N^{2}_{x}}\|\leq \lambda \text{  and also that  }  \|P_{X}^{-1}(X^1(x))|_{N^{1}_{X^1(x)}}\|  \leq \lambda.
\end{equation}
When $\Lambda$ is compact this definition is equivalent to the usual definition of hyperbolic flow (\cite[Proposition 1.1]{D}).

Let us recall that a periodic point $p$ of period $\pi$ is said to
be \emph{hyperbolic} if the linear Poincar\'e flow $P_X^\pi(p)$
has no norm one eigenvalues. We say that $p$ has \emph{trivial
real spectrum} if $P_X^\pi(p)$ has only real eigenvalues of equal
norm to one
  and there exists $0\leq k \leq n-1$
such that 1 has multiplicity $k$ and $-1$ has multiplicity $n-1-k$. Observe that, in the incompressible case, having trivial real spectrum is equivalent to the eigenvalues are equal to 1 or $-1$.

Within Hamiltonian flows we define $\mathcal{N}_x:=N_x\cap T_xH^{-1}(\left\{e\right\})$, where $T_xH^{-1}(\left\{e\right\})=Ker\: \nabla H_x$ is the tangent space to the energy level set. As a consequence we get that $\mathcal{N}_x$ is a ($2d-2$)-dimensional subbundle. The \textit{transversal linear Poincar\'{e} flow} associated to $H$ is given by
\begin{align}
\Phi_H^t(x): &\ \ \mathcal{N}_{x}\rightarrow \mathcal{N}_{X_H^t(x)}\nonumber\\
&\ \ v\mapsto \Pi_{X_H^t(x)}\circ D_x{X_H}^t(v),\nonumber
\end{align}
where $\Pi_{X_H^t(x)}: T_{X_H^t(x)}M\rightarrow \mathcal{N}_{X_H^t(x)}$ denotes the canonical orthogonal projection.

Analogously to was we did for the linear Poincar\'e flow in (\ref{lpf}) we define hyperbolicity and also dominated splitting for the transversal linear Poincar\'{e} flow $\Phi^t_H$. We say that a compact, $X_H^t$-invariant and regular set $\Lambda\subset M$ is \emph{partially hyperbolic} for the transversal linear Poincar\'{e} flow $\Phi^t_H$ if there exists a $\Phi_H^t$-invariant splitting $\mathcal{N}=\mathcal{N}^u\oplus \mathcal{N}^c \oplus \mathcal{N}^{s}$ over $\Lambda$ such that all the subbundles have constant dimension and at least two of them are non-trivial and  $\lambda\in(0,1)$ such that, $\mathcal{N}^u$ is $\lambda$-uniformly hyperbolic and expanding, $\mathcal{N}^s$ is $\lambda$-uniformly hyperbolic and contracting and
 $\mathcal{N}^u$ $\lambda$-dominates $\mathcal{N}^c$ and $\mathcal{N}^c$ $\lambda$-dominates $\mathcal{N}^s$. Along this paper, we consider hyperbolicity, partial hyperbolicity and dominated splitting defined in a set $\Lambda$ which is the whole energy level. It is quite interesting to observe that in Hamiltonians the existence of a do\-mi\-na\-ted splitting implies partial hyperbolicity (see \cite{BV}).

\vspace{0.1cm}

We begin by presenting  our results.
Our  main result in the context of  incompressible flows is a
generalization of the main  result in \cite{Ba} also for higher
dimensional flows.

\begin{theorema} If an incompressible flow $X^t$ is $C^1$-stably average shadowable, or  $C^1$-stably asymptotic average shadowable, then $X^t$ admits a dominated splitting on $M$.
\end{theorema}

\begin{theoremaa}
A $d$-dimensional incompressible flow  ($d \geq 3$)  which is
$C^1$-stably limit shadowable is a  transitive Anosov flow.
\end{theoremaa}

Then, we formulate these results for the Hamiltonian setting.

\begin{theoremb}
A Hamiltonian system $(H,e, \mathcal{E}_{H,e})$
 $C^2$-stably  average
shadowable, or $C^2$-robustly  asymptotic average shadowable, is a
partially hyperbolic Hamiltonian system.
\end{theoremb}

\begin{theorembb}
A Hamiltonian system $(H,e, \mathcal{E}_{H,e})$
 $C^2$-stably  limit
shadowable is Anosov.
\end{theorembb}

Finally, and for the sake of fulfillment of the literature, we obtain a
generalization of the result in \cite[Theorem 4]{AR} for higher
dimensional flows.

\begin{theoremc}
If $X^t$ is $C^1$-stably average shadowable, or $C^1$-stably
asymptotic average shadowable, then $X^t$ admits a dominated
splitting on $M$.
\end{theoremc}

With respect to the limit shadowing case, in \cite[Theorem 1.1]{L2}, Lee recently proved that a  $d$-dimensional  flow  ($d \geq 3$) which is $C^1$-stably limit
shadowable is a transitive  Anosov flow.

We observe that in Theorems A$^\prime$, B$^\prime$ and \cite[Theorem 1.1]{L2} the converse statement is also valid.

The analysis of shadowing for flows is certainly more complicated
than for maps due to presence of reparametrizations of the
trajectories and the (possible) presence of singularities. The
main results in \cite{BC1}  and \cite{LW} can now be obtained from
the Theorem C,  by consideration of suspension flow. Complementing
these result with the other types of systems which we consider
here, we have the following.

\medskip

\begin{corollary}
~
\begin{enumerate}

    \item A   $C^1$-volume preserving
diffeomorphism which is
 $C^1$-stably (asymptotic) average shadowable admits a dominated splitting on
 $M$, and  a $C^1$-stably limit shadowable is transitive Anosov\footnote{The definitions of hyperbolicity and shadowing for the case of diffeomorphisms are analogous to the ones in the case of flows. See \cite{S1}.}.

    \item  A symplectomorphism stably (asymptotic) average
shadowable is partially hyperbolic, and a stably limit shadowable
is Anosov.
\item   A   $C^1$-diffeomorphism  which is
 $C^1$-stably (asymptotic) average shadowable admits a dominated splitting on
 $M$, and  a $C^1$-stably limit shadowable is transitive Anosov.

\end{enumerate}

\end{corollary}

It is noteworthy that, in \cite[Theorem 4]{AR} it is proved that
a $C^1$ stably (asymptotic) average shadowable flow is Anosov
and transitive. In our incompressible $3$-dimensional flow
setting the stability of transitivity is obtained a priori since
by ~\cite{BD} we know that away from Anosov incompressible flows
we have elliptic orbits, thus invariant tori can be created by
small $C^1$-perturbations using \cite{AM}, and so robust
transitivity is not feasible.

We end this introduction by recalling several results in the vein
of ours proved in ~\cite{BR} - \emph{$C^1$-robust topologically
stably incompressible flows are  Anosov}, in ~\cite{F} -
\emph{$C^1$-robust (Lipschitz) shadowing incompressible flows are
Anosov} and in ~\cite{BLV} - \emph{$C^1$-robust weak shadowing
incompressible flows are volume-hyperbolic}. Another result which
relates $C^1$-robust properties with hyperbolicity is the result
in ~\cite{BR3} which states that $C^1$-robustly transitive
incompressible flows have dominated splitting. See also the results in \cite{BRT,BRT2, RRBP} for flows
and in \cite{L,LW} for diffeomorphisms.
\end{subsection}



\section{Incompressible flows - Proofs of Theorems~A and A$^\prime$}\label{diff}

In this section we study the  (asymptotic) average shadowing and
limit shadowing properties in the context of incompressible flows.
This section is divided in two subsections:  the first part is
dedicated to some results about (asymptotic) average shadowing
property which are necessary to prove the Theorem A. The second
subsection is devoted to study the limit shadowing property and
the proof of Theorem A$^\prime$.

\begin{subsection}{(Asymptotic) Average shadowing property - Proof of Theorem A}

We begin by recalling the following result  which was proved in \cite[Proposition 2.4]{BR3} and is the volume-preserving version of \cite[Corollary 2.22]{BGV}:

\begin{proposition}\label{BGV3}
Let $X\in \mathfrak{X}^1_{\mu} (M)$ and fix a small $\epsilon_0 >0$. There exist $\pi_0,\ell\in \mathbb{N}$ such that for any closed orbit $x$ with period $\pi(x)>\pi_0$ we have either
\begin{enumerate}
\item [(i)] that $P_X^t$ has an $\ell$-dominated splitting along the orbit of $x$ or else
\item [(ii)] for any neighborhood $U$ of $\cup_{t}X^t(x)$, there exists an $\epsilon$-$C^1$-per\-tur\-ba\-tion $Y$ of $X$, coinciding with $X$ outside $U$ and on $\cup_{t}X^t(x)$, and such that  $P_Y^{\pi(x)}(x)$ has all eigenvalues equal to $1$ and $-1$.
\end{enumerate}
\end{proposition}

To study perturbations it is convenient to work with linear systems.
Before enunciate the next result let us recall some of these
notions. Given $X\in \mathfrak{X}_\mu^1(M)$ and a regular point
$p$ we consider a \emph{linear differential system} (see
\cite{BR3,BR4} for full details on definitions) over the orbit of
$p$ in the following way:
$$S^t\colon\mathbb{R}^d_{p} \rightarrow \mathbb{R}^d_{X^t(p)}$$
is such that
\begin{itemize}
 \item $S^t \in \text{SL}(d, \mathbb{R})$, for every  $t$;
 \item $S^0=id$ and $S^{t+r}=S^t \circ S^r$, for every $s,t$ and
   \item $S^t$ is differentiable in $t$.
\end{itemize}

In ~\cite[Subsection 2.2]{BR4} it is developed a way to obtain
``good coordinates" adequate to the continuous-time setting.
Actually, we translate, in a conservative fashion, our flow from
the manifold $M$ to the Euclidian space $\mathbb{R}^d$ with
coordinates $(x,y_1,...,y_{d-1})$, and, considering
$v=\frac{\partial}{\partial x}$ we get the following local  linear
representation of the flow $Y$, say valid for $\|(y,z)\|$ very
small;
\begin{equation}\label{GC}
\hat{X}^t((0,y_1,...,y_{d-1}))=tv+S^t(0,y_1,...,y_{d-1}),
\end{equation}
where, in rough terms, $S^t$ represents the action of the linear Poincar\'e flow $P_Y^t$. We observe that, since the flow is linear, the linear Poincar\'e flow equals the Poincar\'e map itself (see \cite{BR4} for full details). Finally, and after performing the perturbations we want, we use the \emph{Pasting Lemma} (see \cite{AM}) to spread (in a volume-preserving way) the linear vector field into a divergence-free vector field that coincides with the original vector field outside a small neighborhood of the periodic orbit.

\begin{lemma}\label{main2}
If $X\in\mathfrak{X}^1_{\mu} (M)$  and  $X^t$ is $C^1$-stably
average shadowable, then any $Y \in \mathfrak{X}^1_{\mu} (M)$
sufficiently $C^1$-close to $X$ does not contains closed orbits
with  trivial real spectrum.
\end{lemma}

\begin{proof}

The proof is by contradiction. Taking into account
Theorem~\ref{sing} let us assume that there exists a $C^1$-stably
average shadowable incompressible flow $X^t$ having a
non-hyperbolic closed orbit $q$ of period $\pi$ and with trivial
real spectrum. Now, we consider a representation of $X^t$, say $\hat{X}^t$, in
the linear coordinates given in (\ref{GC}). Thus, there exists an
eigenvalue, $\lambda$, with modulus equal to one for
$S^{\pi}(\vec{0})$ (where $\vec{0}$ is the image of $q$ in this
change of coordinates). Thus,
\begin{equation}\label{chato5000}\hat{X}^{2\pi}(0,y_1,...,y_{d-1})=2\pi
v+S^{2\pi}(0,y_1,...,y_{d-1})=id\end{equation} holds, say in a
$\xi$-neighborhood of $\vec{0}$. Recall that, since $\hat{X}^t$
has the average shadowing property  $\hat{X}^{2\pi\,t}$ also has.

Take two points $w_1=(0,y_1,...,y_{d-1})$ and $w_2=-w_1$, with
$d(w_1,w_2)=\xi/2$ and take $\epsilon:=d(w_1,w_2)/ 3$.
Given $\delta > 0$, as in the definition of average shadowing,
pick $n_0=n_0(\delta, \epsilon)$ a sufficiently large positive
integer such that ${\xi}/{n_0}< \delta$.

Then, we define a  sequence $(x_i,t_i)_{i\in\mathbb{Z}}$ which
$x_i$ takes the values $w_1$ and $w_2$ alternately by steps of
length $2 ^{j}$, in the following way:
\begin{eqnarray}\label{rrbp}
x_i &=&  w_1, \,\,\,t_i = 1\,\,\,\,\text{if } 2^{2j} \leq i < 2^{2j+1}~,\,\,\,\,\,\,\,\,\,\,\,\,\,\,\\
x_i &=& w_2,\,\,\,t_i = 1\,\,\,\,\text{if } 2^{2j+1}\leq i <
2^{2(j + 1)} \nonumber
\end{eqnarray} for $j\in\mathbb{Z}$.

The sequence $(x_i,t_i)_{i\in\mathbb{Z}}$  is a
 $\delta$-average-pseudo-orbit of $\hat{X}^{2\pi}$. Indeed, take   $m\in \mathbb{N}$  large enough,
 such that $m > 2^{n_0}$. Then  $m=2^{n} +r$, with $n \geq n_0$, $n \in \mathbb{N}$, and  $r \in \{0, \ldots, 2^{n}-1
 \}$. So,
\begin{eqnarray}\label{chato1000}\frac{1}{m}\sum_{k=1}^{m}
d(\hat{X}^{2\pi(t_{k+i})}(x_{k+i}),x_{k+i+1})&=&\frac{1}{2^{n}+r}\sum_{k=1}^{2^{n}+r}
d(\hat{X}^{2\pi(t_{k+i})}(x_{k+i}),x_{k+i+1}) \\\nonumber &\leq&
\frac{1}{2^{n} }\sum_{k=1}^{2^{n}}
d(\hat{X}^{2\pi(t_{k+i})}(x_{k+i}),x_{k+i+1}) <\dfrac{n}{2^{n}}\cdot\dfrac{\xi}{2}<\dfrac{1}{n}\cdot\dfrac{\xi}{2}<
\delta.
\end{eqnarray}

Denote $\hat{X}^{2\pi}$ by ${X}$. So the $\delta$-pseudo orbit
$(x_i,t_i)_{i\in\mathbb{Z}}$ can be $\epsilon / 2$ positively
shadowed in average by  the orbit of ${X}$ through some point $ z
\in M$, that is, there is an orientation preserving homeomorphism
$h\colon \mathbb{R}\rightarrow \mathbb{R}$ with $h(0)=0$ such that
\begin{equation}\label{rr22}
    \underset{n\rightarrow \infty}{\limsup}\frac{1}{n}\sum_{i=1}^n \int_{i}^{{i+1}} d(X^{h(t)}(x),
    X^{t-i}(x_i))dt<\epsilon/2.
\end{equation}
We can assume without loss of generality that
$d(z,w_1)<{\epsilon}/{2}$. Thus,
$$d(X^{h(t)}(z),X^{t-i}(w_2))>2\epsilon \;\; \textmd{for}\; \textmd{all} \;\; t \in \mathbb{R}.$$

\noindent This implies that for every $n> n_0$ enough large,
$$ \frac{1}{n}\sum_{i=1}^n \int_{i}^{{i+1}} d(X^{h(t)}(z),
    X^{t-i}(x_i))dt > {\epsilon},$$

     \noindent and therefore, $$ \underset{n\rightarrow \infty}{\limsup}\frac{1}{n}\sum_{i=1}^n \int_{i}^{{i+1}} d(X^{h(t)}(z),
    X^{t-i}(x_i))dt > \frac{\epsilon}{2}.$$

\noindent This contradicts (\ref{rr22}).
\end{proof}

\begin{lemma}\label{socorro500}
If $X\in\mathfrak{X}^1_{\mu} (M)$ and $X^t$ is $C^1$-stably
asymptotic average shadowable, then any $Y \in
\mathfrak{X}^1_{\mu} (M)$ sufficiently $C^1$-close to $X$ does not
contains closed orbits with  trivial real spectrum.
\end{lemma}

\begin{proof} The proof follows analogously to  Lemma \ref{main2}. Suppose  that
 there exists a $C^1$-stably asymptotic average shadowable incompressible flow $X^t$ having a  non-hyperbolic closed orbit $q$ of period $\pi$ and with trivial real
 spectrum, since, by Theorem~\ref{sing}, the vector field $X$ has no
 singularities.

Consider $\hat{X}^t$   a representation of $X^t$. Then, by
(\ref{chato5000}), $\hat{X}^{2\pi t} = id $ in a
$\xi$-neighborhood of $\vec{0}$. Recall that, since $\hat{X}^t$
has the asymptotic average shadowing property, $\hat{X}^{2\pi\,t}$
also has.

Take two points $w_1=(0,y_1,...,y_{d-1})$ and $w_2=-w_1$, with
$d(w_1,w_2)=\xi/2$ and take $\epsilon:=d(w_1,w_2)/ 3$.
Consider $(x_i,t_i)_{i\in\mathbb{Z}}$  the sequence defined in
(\ref{rrbp}). Observe that the  sequence is a asymptotic average
pseudo orbit. In fact, for $m \in \mathbb{N}$ large enough, the
inequality (\ref{chato1000}) implies that,
\begin{equation*}
    \lim_{m\to \infty}\frac{1}{m}\sum_{k=1}^m d(X^{t_k}(x_k),x_{k+1}) = 0.
\end{equation*}

Denote $\hat{X}^{2\pi}$ by $X^1$. So the asymptotic average pseudo
orbit $(x_i,t_i)_{i\in\mathbb{Z}}$ can be $\epsilon / 2$
positively asymptotic shadowed in average by  the orbit of ${X}$
through some point $ z \in M$, that is, there is an orientation
preserving homeomorphism $h\colon \mathbb{R}\rightarrow
\mathbb{R}$ with $h(0)=0$ such that
\begin{equation}\label{chato6000}
    \underset{n\rightarrow \infty}{\lim}\frac{1}{n}\sum_{i=1}^n \int_{i}^{{i+1}} d(X^{h(t)}(z),
    X^{t-i}(x_i))dt = 0.
\end{equation}
We can assume without loss of generality that
$d(z,w_1)<{\epsilon}/{2}$. This implies that for  $n$ sufficiently
big,
$$ \frac{1}{n}\sum_{i=1}^n \int_{i}^{{i+1}} d(X^{h(t)}(z),
    X^{t-i}(x_i))dt > {\epsilon}.$$

\noindent Therefore,

\begin{equation*}
    \underset{n\rightarrow \infty}{\lim}\frac{1}{n}\sum_{i=1}^n \int_{i}^{{i+1}} d(X^{h(t)}(z),
    X^{t-i}(x_i))dt > 0,
\end{equation*}

\noindent which contradicts (\ref{chato6000}).\end{proof}


\begin{subsection}{Rule out singularities}

In order to prove that there are no coexistence of singularities
and the $C^1$-stably (asympotic) average shadowable
incompressible flows we will recall some useful results.

\begin{lemma}\label{linear}\cite[Lemma 3.3]{BR3}
Let $\sigma$ be a singularity of $X \in \mathfrak{X}_\mu^1(M)$. For any $\epsilon >0$ there exists  $Y \in  \mathfrak{X}_\mu^{\infty}(M)$, such that $Y$ is $\epsilon$-$C^1$-close to $X$ and  $\sigma$ is a linear hyperbolic singularity of $Y$.
\end{lemma}

The second one, was proved in ~\cite[Proposition 4.1]{V} generalizing the Doering theorem in \cite{D}. Observe that, in our volume-preserving context, the singularities of hyperbolic type are all saddles.

\begin{proposition}\label{Vivier}
If $~Y \in \Cmundo$ admits a linear hyperbolic singularity of
saddle-type, then the linear Poincar\'e flow of $Y$ does not admit
any dominated splitting  over $M\setminus Sing(Y)$.
\end{proposition}

Finally, since by Poincar\'e recurrence, any $X \in \mathfrak{X}_\mu^1(M)$ is chain transitive, the following result is a direct consequence of \cite{Be}.

\begin{proposition}\label{Mixing}
In $\mathfrak{X}_\mu^1(M)$ chain transitive flows are equal
topologically mixing flows in a $C^1$-residual subset.
\end{proposition}

The following theorem is proved using an analog reasoning as the one in
\cite[Theorem 15]{AR}.

\begin{theorem}\label{sing}
If $X^t$ is a  $C^1$-stably (asymptotic) average shadowable
incompressible flow, then $X^t$ has no singularities.
\end{theorem}
\begin{proof}
Let $X\in\mathfrak{X}_\mu^1(M)$ be  with  $X^t$ $C^1$-stably
(asymptotic) average shadowable. Fix a small $C^1$ neighborhood
$\mathcal{U}\subset \mathfrak{X}_\mu^1(M)$ of $X$. The proof is by
contradiction. Assume that $Sing(X)\not= \emptyset$. Using
Lemma~\ref{linear}, there exists $Y\in \mathcal{U}$ with a linear
saddle-type singularity $\sigma\in Sing(Y)$. By
Proposition~\ref{Mixing}, there exist $Z_n\in
\mathfrak{X}_\mu^1(M)$ $C^1$-close to $Y$ which is topologically
mixing. We can find $W_n\in \mathfrak{X}_\mu^1(M)$ $C^1$-close to
$Z_n$ having a $W_n$-closed orbit $\mathcal{O}(p_n)$ such that the
Hausdorff distance between $M$ and $\cup_t
W_n^t(\mathcal{O}(p_n))$ is less than $1/n$.

Now we consider jointly Proposition~\ref{BGV3} and  Lemma~\ref{main2}
(Lemma~\ref{socorro500}) and obtain that $P_{W_n}^t$ is
$\ell$-dominated over the $W_n$-orbit of $\mathcal{O}(p_n)$ where
$\ell$ is uniform on $n$. Since $W_n$ converges in the $C^1$-sense
to $Y$ and $\lim \sup_n \cup_t W_n^t(\mathcal{O}(p_n))=M$ we
obtain that $M\setminus Sing(Y)$ has an $\ell$-dominated splitting
which contradicts Proposition~\ref{Vivier}.
\end{proof}

\end{subsection}


The proof of Theorem A  similarly follows the proof of the Theorem
\ref{sing}.

\begin{proof}(of Theorem A)

Let $X\in \Cmundo$ be a vector field with $X^t$ $C^1$-stably
average shadowable ($C^1$-stably asymptotic average shadowable)
flow and fix a small $C^1$ neighborhood $\mathcal{U}\subset
\mathfrak{X}_\mu^1(M)$ of $X$. By Proposition~\ref{Mixing}, there
exists $Z_n\in \mathfrak{X}_\mu^1(M)$ $C^1$-close to $X$ which is
topologically mixing. We can find $W_n\in \mathfrak{X}_\mu^1(M)$
$C^1$-close to $Z_n$ having a $W_n$-closed orbit
$\mathcal{O}(p_n)$ such that the Hausdorff distance between $M$
and $\cup_t W_n^t(\mathcal{O}(p_n))$ is less than $1/n$.

Now, we consider  together Proposition ~\ref{BGV3}, Lemma ~\ref{main2}
(Lemma~\ref{socorro500}) obtaining that $P_{W_n}^t$ is
$\ell$-do\-mi\-nated over the $W_n$-orbit of $\mathcal{O}(p_n)$
where $\ell$ is uniform on $n$. Since $W_n$ converges in the
$C^1$-sense to $X$ and
$$\lim \sup_n \bigcup_t
W_n^t(\mathcal{O}(p_n))=M$$ we obtain that $M\setminus Sing(X)$
has an $\ell$-dominated splitting. Now, if $X$ has a singularity
$\sigma$, then using Lemma~\ref{linear} there exists $Y\in
\Cmundo$ $C^1$-close to $X$ with a linear saddle-type singularity,
and proceeding as above $P^t_Y$ admits a $\ell$-dominated
splitting over $M \backslash Sing(Y)$, which contradicts
Proposition~\ref{Vivier}. Therefore, $M$ admits an
$\ell$-dominated splitting.\end{proof}
\end{subsection}


\begin{subsection}{Limit shadowing property - Proof of Theorem A$^\prime$}\label{quesaco21}

We
recall that by the stable manifold theory (cf.~\cite{ss39}), if
$\mathcal{O}$ is a hyperbolic closed orbit of $X$
 with splitting
 $T_{\mathcal{O}}M=E_{\mathcal{O}}^s \oplus E^X_{\mathcal{O}}\oplus
 E_{\mathcal{O}}^u$
then its \emph{unstable set}
$$ W^u(\mathcal{O})=\{ y\in M; \alpha(y)=\mathcal{O}\},$$
 \noindent is an immersed submanifold tangent at
 $\mathcal{O}$ to the subbundle  $E^X_{\mathcal{O}}\oplus
 E_{\mathcal{O}}^u$, and its \emph{stable set}
$$ W^s(\mathcal{O})=\{ y\in M; \omega(y)=\mathcal{O}\},$$
\noindent is an immersed submanifold tangent at
 $\mathcal{O}$ to the subbundle $E_{\mathcal{O}}^s\oplus
 E^X_{\mathcal{O}}$. In this case
  $W^{s}(\mathcal{O})$ and $W^{u}(\mathcal{O})$ are called the
  \emph{stable} and  the \emph{unstable} \emph{manifolds} of
  $\mathcal{O}$, respectively.

We observe that  the following lemma stated for dissipative flows holds in this setting of incompressible flows.

\begin{lemma}\cite[Lemma 23]{RRBP}\label{rr30B} If $X$ has the
limit shadowing property then $W^s(\mathcal{O}) \cap
W^u(\mathcal{O}') \neq \emptyset$ for any pair of orbits
$\mathcal{O}, \mathcal{O}' \in \Crit(X)$.
\end{lemma}

We start recalling that for vector fields in $\Cmundo$,
whose dimension $d$ of $M$ is greater than or equal to 3, it is
proved in  \cite{Roro}  the existence of a residual subset of
vector fields such that every singularity and periodic orbit is
hyperbolic (or elliptic, if $d$ = 3), and the corresponding
invariant manifolds intersect transversely.  We have also a  kind
of Franks' lemma for conservative flows \cite[Lemma 3.2]{BR3}.

\begin{lemma} \label{quesaco00} If $X^t$  is a incompressible flow which is $C^1$-stably limit shadowable, then
all periodic orbits of $X$ are hyperbolic.
\end{lemma}
\begin{proof}
\textbf{Case 1:} $\dim(M)= 3$. Let $\mathcal{U}$ be a
neighborhood of $X$ as definition of robustness of the limit
shadowing property. Let $p$ be a periodic point of $X$, and
suppose $\mathcal{O}$($p$) is non-hyperbolic. Using the \cite[Theorem 1.3]{BD} there exists $Y$ in
$\mathcal{U}$ with a elliptic periodic orbit $\mathcal{O}$($p_1$). Now, using Franks' lemma (\cite[Lemma 3.2]{BR3}) we can assume that for a perturbation in $\mathcal{U}$ the linear Poincar\'e flow $P$ on the period of the analytic continuation of $p_1$ is a rational rotation. Hence, the argument in  (\cite[Lemma 4.3]{AM}) allows us to obtain a tubular neighborhood of the periodic orbit. Clearly, there exists $n\in\mathbb{N}$ such that $P^n=id$ but this, due to a criteria in ~\cite{Pl}, invalidates the possibility of having the limit shadowing property (recall that if $P^n$ has the limit shadowing property, then $P$ also has).\\
\textbf{Case 2:} $\dim(M)\geq 3$.
Let $\mathcal{U}$ be a neighborhood of $X$ such that the limit
shadowing property holds. Since by the general density theorem (\cite{PR}), $C^1$-generically, we have plenty of periodic points we let $p$ and $q$ be periodic points of
$X$. We claim that any periodic point is hyperbolic. We assume by contradiction that $\mathcal{O}(p)$ (or $\mathcal{O}(q)$ of both) is a non-hyperbolic periodic
orbit of $X$. By \cite[Lemma 3.2]{BR3}, there exists $Y$
$C^1$-close to $X$ in $\mathcal{U}$ with $\mathcal{O}(p_1)$ and $\mathcal{O}(q_1)$
hyperbolic periodic orbits of $Y$ with different index (dimension of the stable manifold). Since we can consider $Y$ a
Kupka-Smale vector field (\cite{Roro}), we have that
$W^s(\mathcal{O}(p_1))\cap W^u(\mathcal{O}(q_1) )= \emptyset$ or $W^u(\mathcal{O}(p_1))\cap W^s(\mathcal{O}(q_1) )= \emptyset$. This
contradicts  Lemma \ref{rr30B} for incompressible flows, and prove
the desired.
\end{proof}

Now, we obtain  the following
result.

\begin{lemma} \label{quesaco22} An incompressible flow $X^t$ which is $C^1$-stably limit shadowable
has no singularities.
\end{lemma}

\begin{proof}
Suppose that $X$ has a
singularity $\sigma_0$, then by \cite{Roro,PR} there exists
$Y\in \mathfrak{X}_\mu^1(M)$ $C^1$-close to $X$ with a hyperbolic singularity ${\sigma}$
and a hyperbolic periodic orbit ${\gamma}$ of different indices
\emph{i} and \emph{j}, respectively.\\
If j $<$ i, then
$$\dim W^u(\sigma) + \dim W^s(\gamma)=(\dim M-i)+j \leq
\dim \emph{M}.$$
Since we can consider $Y$ a Kupka-Smale (Robinson) vector field we
have $\dim W^u(\sigma)+ \dim W^s(\gamma) =  \dim M$. By
Lemma \ref{rr30B}, we can consider $x \in W^u (\sigma)\cap
W^s(\gamma)$. Then $\mathcal{O}$(\emph{x})$ ~\subset W^u
(\sigma)\cap W^s(\gamma)$ and we can split
$$ T_x(W^u(\sigma)) = T_x (\mathcal{O}(x)) \oplus E^1  \;\
\textmd{and} \;\  T_x(W^s(\gamma)) = T_x (\mathcal{O}(x)) \oplus
E^2 .$$ So, $\dim ( T_x(W^u(\sigma))+ T_x (W^s(\gamma))) <
\dim W^u (\sigma) + \dim W^s(\gamma) =
\dim M$.  This is a contradition, because \emph{X} is a
Kupka-Smale vector field.\\
If  \emph{j} $\geq$ \emph{i}, then $\dim W^s(\sigma) + \dim
W^u(\gamma) \leq \dim M$ and by the same arguments we have a
contradiction. Thus $X$ has no singularities.\end{proof}

We recall that an incompreensible flow is said to be an
\emph{incompreensible star flow} if there exists a
$C^1$-neighborhood $\mathcal{U}$ of $X$ in $\Cmundo$ such that any
critical orbit of any  $Y \in \mathcal{U}$ is hyperbolic. A
consequence of Lemmas \ref{quesaco00} and  \ref{quesaco22} is the
following result.

\begin{corollary}\label{quesaco20} If $X \in \mathfrak{X}_\mu^1(M)$ is $C^1$-stably limit shadowing
shadowable then $X^t$ is a star flow without singularities.
\end{corollary}

\begin{theorem}\label{sono02} \cite[Theorem 1]{CF} If $X\in \Cmundo$ is a star flow
without singularities, then $X^t$ is a transitive Anosov flow.
\end{theorem}

The proof of Theorem A$^\prime$ follows directly by Corollary
\ref{quesaco20} and Theorem \ref{sono02}.

\end{subsection}


\section{Hamiltonian flows-proof of Theorems~ B and B$^\prime$}

Observe that, in the Hamiltonian context, we only consider regular energy surfaces, thus we do not have to deal with singularities.  Now, we
state the following Hamiltonian version of \cite[Corollary 2.22]{BGV} and Proposition~\ref{BGV3} and was proved in \cite[Theorem 3.4]{BRT}:

\begin{theorem}\label{main}
Let $H\in C^2(M,\mathbb{R})$ and $\mathcal{U}$ be a neighborhood
of $H$ in the $C^2$-topology. Then for any $\epsilon>0$ there are
$\ell,\pi_0\in\mathbb{N}$ such that, for any $H_0\in \mathcal{U}$
and for any periodic point $p$ of period $\pi(p)\geq \pi_0$:
\begin{enumerate}
\item either $\Phi^t_{H_0}(p)$ admits an $\ell$-partially hyperbolic splitting along the orbit of $p$;
\item or, for any tubular flowbox neighborhood $\mathcal{T}$ of the orbit of $p$, there exists an $\epsilon$-$C^2$-perturbation $H_1$ coinciding with $H_0$ outside $\mathcal{T}$ and whose transversal linear Poincar\'e flow $\Phi^{\pi(p)}_{H_1}(p)$ has all eigenvalues with modulus equal to $1$.
\end{enumerate}
\end{theorem}

\begin{proof}(of Theorem B) By \cite{BFR2}, exist $H_n\in C^2(M,\mathbb{R})$
$C^2$-close to $H$ and $\tilde{e}$ arbitrarily close to $e$, such
that $\mathcal{E}_{H_n,\tilde{e}}$ is topologically mixing. We can
find $\hat H_n\in C^2(M,\mathbb{R})$ $C^2$-close to $H_n$ and
$\hat e$ close to $e$ having a $X_{\hat H_n}^t$-closed orbit $p_n$
such that the Hausdorff distance between $\mathcal{E}_{\hat
H_n,\hat{e}}$ and $\bigcup_t X_{\hat H_n}^t(p_n)$ is less than
$1/n$.

Then, we follow the steps of Lemma~\ref{main2}
(Lemma~\ref{socorro500}) but using the formalism developed in
\cite[Lemma 6.1]{BRT2} obtaining that $P_{\hat H_n}^t$ is
$\ell$-dominated over the $X_{\hat H_n}^t$-orbit of $p_n$ where
$\ell$ is uniform on $n$. Since $\hat H_n$ converges in the
$C^2$-sense to $H$ and
$$\lim \sup_n \bigcup_t
X_{\hat H_n}^t(p_n)=\mathcal{E}_{\hat H_n,\hat{e}}$$ we obtain that $\mathcal{E}_{\hat H_n,\hat{e}}$ has an
$\ell$-dominated splitting. Therefore, $\mathcal{E}_{\hat H_n,\hat{e}}$ admits an
$\ell$-dominated splitting. By \cite[Remark 2.1]{BRT2} $\mathcal{E}_{\hat H_n,\hat{e}}$ is partial hyperbolic. Finally, we observe that partial hyperbolicity spreads to the closure and we are over.
\end{proof}

\subsection{Limit shadowing property - Proof of Theorem B$^\prime$.}
 Now, we prove  Theorem
B$^\prime$. The proof of this theorem follows as in Theorem A$^\prime$. So, we
only mention the version for Hamiltonian systems  of results in
Subsection \ref{quesaco21}, and the modifications necessary to
obtain our result.

Firstly, note that Lemma \ref{rr30B} is true for Hamiltonian
systems. Now, we need to check the Lemmas \ref{quesaco00} and
\ref{quesaco22}
 in this context. For this we recall the notions
of Kupka-Smale Hamiltonian systems. A Hamiltonian system $(H,e,
\mathcal{E}_{H,e})$ is a \emph{Kupka-Smale Hamiltonian system} if
the union of the hyperbolic and $k$-elliptic closed orbits $( 1
\leq k \leq n-1)$ in $\mathcal{E}_{H,e}$ is dense in
$\mathcal{E}_{H,e}$ and the intersection of invariant of the
closed orbits intersect transversally. Furthermore, the Kupka-Smale Hamiltonian systems form a residual in $C^2(M,\mathbb{R})$.
See \cite[Theorem 1 and 2]{Roro}.

Now,  we  state Lemma \ref{quesaco00} for Hamiltonian systems.

\begin{lemma}\label{acabou01} If a Hamiltonian system $(H,e, \mathcal{E}_{H,e})$ is $C^2$-stably limit shadowable, then
all its  periodic orbits are hyperbolic.
\end{lemma}

\begin{proof}

\textbf{Case 1:} $\dim(M)= 4$. Let $\mathcal{U}$ be a
neighborhood of $H$ as definition of robustness of the limit
shadowing property. Let $p$ be a periodic point of $(H,e, \mathcal{E}_{H,e})$, and
suppose $\mathcal{O}$($p$) is non-hyperbolic. Using \cite{BeD} there exists $H_0$ in
$\mathcal{U}$ with an elliptic periodic orbit $\mathcal{O}$($p_1$). Now, using Franks' lemma for Hamiltonians (\cite{AD}) we can assume that for a perturbation in $\mathcal{U}$ the transversal linear Poincar\'e flow $P$ on the period of the analytic continuation of $p_1$ is a rational rotation. Hence, the argument similar to the one in  (\cite[Lemma 4.3]{AM}) allows us to obtain a tubular neighborhood of the periodic orbit. Clearly, there exists $n\in\mathbb{N}$ such that $P^n=id$ but this, due to a criteria in ~\cite{Pl}, invalidates the possibility of having the limit shadowing property.

\textbf{Case 2:} $\dim(M)=2d\geq 6$.
Let $\mathcal{U}$ be a neighborhood of $H$ such that the limit
shadowing property holds. Since by the general density theorem (\cite{PR}), $C^1$-generically, we have plenty of periodic points we let $p$ and $q$ be periodic points of
$(H,e, \mathcal{E}_{H,e})$. Clearly, we can assume that one of them, say $q$, is hyperbolic. We claim that $p$ must be also  hyperbolic. We assume by contradiction that $\mathcal{O}(p)$ is a non-hyperbolic periodic
orbit of $(H,e, \mathcal{E}_{H,e})$. Thus, generically $\mathcal{O}(p)$ is elliptic or $k$-elliptic (partial hyperbolic). If $p$ is partially hyperbolic, then $p$ and $q$ have different index. Since we can consider a
Kupka-Smale vector field (\cite{Roro}) $C^2$-near $H$, in $\mathcal{U}$ and still with the analytic continuations of $p$ and $q$ respectively, partial hyperbolic and hyperbolic, we have that
$W^s(\mathcal{O}(p))\cap W^u(\mathcal{O}(q) )= \emptyset$ and $W^u(\mathcal{O}(p_1))\cap W^s(\mathcal{O}(q_1) )= \emptyset$. This
contradicts  Lemma \ref{rr30B} for Hamiltonian flows, and prove
the desired. Finally, if $p$ is elliptic, then the argument in Case 1 but with several rational rotations in the symplectic subspaces allows us to obtain a contradiction.
\end{proof}

To complete the proof we recall the notion of
star Hamiltonian system. A  Hamiltonian systems $(H,e,
\mathcal{E}_{H,e})$ is a \emph{star Hamiltonian system} if there
exists a neighborhood $\mathcal{V}$ of $(H,e, \mathcal{E}_{H,e})$
such that, for any $(\widetilde{H}, \widetilde{e} ,
\mathcal{E}_{\widetilde{H},\widetilde{e}}) \in \mathcal{V}$, the
correspondent regular energy hypersurface
$\mathcal{E}_{\widetilde{H},\widetilde{e}}$ has all the critical
orbits hyperbolic.

As consequence of this results for Hamiltonin systems, we obtain
the following result.

\begin{corollary}If $(H,e, \mathcal{E}_{H,e})$ is a Hamiltonian system
 $C^2$-robustly  limit
shadowable, then $(H,e, \mathcal{E}_{H,e})$ is a Hamiltonian star
systems.
\end{corollary}

In \cite[Theorem 1]{BRT} is was proved that a Hamiltonian star
system, defined on a $2d$-dimensional ($d\geq 2$) is Anosov. This
concluded the proof of Theorem B$^\prime$.

\section{Dissipative flows - proof of Theorem~C}\label{diff}


Firstly, we recall that $X^t$ is \emph{chain transitive}, if for any
points $ x, y \in M$ and any $\delta
> 0$, there exists a finite $\delta$-pseudo orbit
$(x_i,t_i)_{0\leq i \leq K}$ of $X$  such that $x_0=x$ and
$x_K=y$.  Observe that transitivity implies chain transitivity. We recall also that  a vector field $X$ has a \emph{property
$\mathfrak{P}$ robustly }if
 there  exists a $C^1$-neighborhood $\mathcal{U}$ of $X$
 such that any $Y\in \mathcal{U}$ has the property $\mathfrak{P}$. Theorem C is a direct consequence of the more general result:

\begin{theorem} \label{general1}If $X \in \Mundo$  is robustly chain transitive,
then $X^t$ admits a dominated splitting on $M$.
\end{theorem}

 A compact
invariant set $\Lambda$ is \emph{attracting} if $\Lambda =
\bigcap_{t\geq0} X^t(U)$ for some neighborhood $U$ of $\Lambda$
satisfying, $X^t(U) \subset U$ for all $t>0$. An \emph{attractor}
of $X$ is a transitive attracting set of $X$ and a \emph{repeller}
is an attractor for $-X$. We say that $\Lambda$ is a \emph{proper}
attractor or repeller if $\emptyset\neq\Lambda\neq M$. A
\emph{sink} (\emph{source}) of $X$ is a attracting (repelling)
critical orbit of $X$.

We recall that  the chain-transitivity rules out the presence of
sinks and sources \cite[Lemma 6]{AR} and that robustly chain
transitive vector fields have no singularities \cite[Theorem
15]{AR}. We also recall that the \textit{Hausdorff distance}
between two compact subsets $A$ and $B$ of $M$ is defined by:
$$d_H(A,B)=\max\{\sup_{x\in A}d(x,B),\sup_{y\in B}d(y,A)\}.$$
We also make use of the following result.

\begin{theorem}\cite[Theorem 4]{C}
\label{t.crovisier} There exists a residual set $\mathcal{R}$ of
$\Mundo$ such that for any vector field $X\in \mathcal{R}$, a
compact invariant set $\Lambda$ is the limit, with respect to the Hausdorff
distance, of a sequence of periodic orbits if and only if $X$ is
chain transitive in $\Lambda$.
\end{theorem}

Finally, the dichotomy in \cite[Corollary 2.22]{BGV} will play an important role along our proof.

To prove Theorem C  is sufficient to prove Theorem
\ref{general1}, since vector fields with (asymptotic) average
shadowing property are chain transitive (cf. \cite{GSX3, GSX}).
The proof of Theorem \ref{general1} follows in a manner analogous
to \cite[Theorem 15]{AR}.

\begin{proof}(Proof of Theorem \ref{general1}) Let $X$ be a robustly chain transitive vector field, $\mathcal{U}$
the neighborhood of $X$ as in the definition, and $\mathcal{R}$
the residual in $\mathcal{U}$ of the Theorem \ref{t.crovisier}. Hence, there exist a sequence  $Y_n\in \mathcal{R}$, converging to
$X$,  and periodic orbits $O_{Y_n}(p_n)$ of $Y_n$ such that
$M=\limsup O_{Y_n}(p_n)$. As $Y_n$ neither admits sinks nor
sources, \cite[Corollary 2.22]{BGV} assure that the linear Poincar\'e
flow $P_{Y_n^t}^t$ admits an $\ell$-dominated splitting over
$O_{Y_n}(p_n)$, with $\ell$ independent of $n$. Therefore,
$P_{X^t}^t$ admits an $\ell$-dominated splitting over $M$.
\end{proof}



\section*{Acknowledgements}
The first author would like to thank Alexander Arbieto for suggestions given during the preparation of this work.


\end{document}